\documentclass[graybox]{svproc}

\usepackage{diagbox}
\usepackage{mathptmx}       
\usepackage{helvet}         
\usepackage{courier}        
\usepackage{type1cm}        
\usepackage{makeidx}         
\usepackage{graphicx}        
\usepackage{multicol}        
\usepackage[bottom]{footmisc}
\usepackage{amsmath}
\usepackage{amssymb,amsfonts}
\usepackage{subfig}

\usepackage{url}

\begin{document}

\mainmatter              
\title{Convergence of the Neumann-Neumann Method for the Cahn-Hilliard Equation}
\titlerunning{NN for the CH equation}
\author{Gobinda Garai \inst{1}}
\institute{
School of Basic Sciences (Mathematics), Indian Institute of Technology Bhubaneswar, India.
\email{gg14@iitbbs.ac.in}}

\maketitle

\abstract{In this paper, we analyze a substructuring type algorithm for the Cahn-Hilliard (CH) equation. Being a nonlinear equation, it is of great importance to develop efficient numerical schemes for investigating the solution behaviour of the CH equation. We present the formulation of  Neumann-Neumann (NN) method applied to the CH equation and investigate the convergence behaviour of the same in one and two spatial dimension for two subdomains. We illustrate the theoretical results by providing numerical example.}


\section{Introduction}
The Cahn-Hilliard equation
\begin{equation}\label{mixedCH}
\begin{aligned}
\frac{\partial u}{\partial t} & = \Delta v, \,\ \mbox{ for} \,\ (x,t)\in\Omega\times(0, T],\\
v & = F'(u) - \epsilon^2\Delta u, \,\ \mbox{ for} \,\ (x,t)\in\Omega\times(0, T],\\
u(x, 0) & = u_0(x), \,\ x\in\Omega,
\end{aligned}
\end{equation}
where $\Omega \subset\mathbb{R}^d (d=1,2)$, with the Neumann boundary condition $\frac{\partial u}{\partial n} = \frac{\partial v}{\partial n} = 0 \mbox{ on} \,\ \partial\Omega,$ for all $t\in(0, T]$,
has been suggested to represent the evolution of a binary melted alloy below the critical temperature in \cite{Cahn,Hilliard}. The solution of \eqref{mixedCH} involves two different dynamics, one is phase separation  which is quick in time, and another is phase coarsening which is relatively  slower in time. The fine-scale phase regions are emerged in the early period and  separated by the interface which is of width $\epsilon$. Whereas during phase coarsening, the solution lean toward an equilibrium state which reduces the internal energy. This means, the equation \eqref{mixedCH} describes the total mass conservation and energy minimization while the system evolves.
The energy minimization is a key feature of \eqref{mixedCH} and is expected to be preserved by numerical method and to deal with that, Eyre \cite{David,Eyre} proposed an unconditionally gradient stable scheme. The idea is to split the homogeneous free energy $F(u) = 0.25(u^2 - 1)^2$ into the sum of a convex and a concave term, and then treating the convex term implicitly and the concave term explicitly to obtain, for example, a first order in time and 2nd order in space approximation for \eqref{mixedCH} in 1D:
\begin{equation}\label{1stchdis}
\begin{aligned}
u_j^{n+1} & = \delta\Delta _hv_j^{n+1} + u_j^{n},\\
v_j^{n+1} & = (u_j^{n+1})^3 -u_j^{n} - \epsilon^2\Delta _hu_j^{n+1}, 
\end{aligned}
\end{equation}
where $\delta$ is the time step and $\Delta_h$ is the discrete Laplacian.
The equation \eqref{1stchdis} represents a system of nonlinear coupled equation due to the cubic term. To linearise the problem, the term $(u_j^{n+1})^3$ is rewritten as $(u_j^{n})^2u_j^{n+1}$, which results the following modified system 
\begin{align*}
u_j^{n+1} - \delta\Delta _hv_j^{n+1} & = u_j^{n},\\
v_j^{n+1} + \epsilon^2\Delta _hu_j^{n+1} - (u_j^{n})^2u_j^{n+1} & = -u_j^{n}, 
\end{align*}
which is also an unconditionally gradient stable scheme and has the same accuracy as the nonlinear scheme \eqref{1stchdis} \cite{David}. Therefore one has to solve the following system of elliptic equations at each time level
\begin{equation}\label{MF1}
\begin{aligned}
\begin{bmatrix} 
I & -\delta\Delta \\
\epsilon^2\Delta-c^2 & I 
\end{bmatrix}
\quad
\begin{bmatrix} 
{\bar{u}} \\
{\bar{v}} 
\end{bmatrix}
\quad 
 = \quad \begin{bmatrix} 
f_{\bar{u}} \\
f_{\bar{v}} 
\end{bmatrix},
\quad \mbox{in}\,\ \Omega,
\end{aligned}
\end{equation}
where $c = u_j^{n}, \bar{u} = u_j^{n+1}, \bar{v} = v_j^{n+1}, f_{\bar{u}} = c \,\ \text{and}\,\ f_{\bar{v}} = -c.$
In 2D, one also gets the above system \eqref{MF1} at each time level for suitably chosen $f_{\bar{u}}, f_{\bar{v}}, c$. It is worth mentioning that many other basic algorithms approximating the solution of the CH equation can be reformulated as \eqref{MF1}, for example the semi-implicit Euler's scheme \cite{KimLee}, the Linearly stabilized splitting scheme \cite{EsedoAndrea}. 

Since the spatial mesh size $h$ is $O(\epsilon)$ or even finer, the linear system \eqref{MF1} will result in a massive algebraic structure and which needs to be solved serially for getting the long term solution of \eqref{mixedCH}. Thus, it is of great importance to fast-track the simulation using parallel computation, which can be fulfilled by domain decomposition techniques \cite{LionsI,TosWid}. In this work, we lay our efforts on the Neumann-Neumann method \cite{bourgat1988variational,le1991domain,BjWid}.
The main objective of our work is to solve the problem \eqref{MF1} with the imposed transmission condition and analyze the convergence  behaviour for two subdomain setting in 1D and 2D.

We introduce the NN algorithm in one and two spatial dimension for two subdomains, and state our convergence result in Section \ref{Section2}. To illustrate our analysis, the accuracy and robustness of the proposed formulation, we present the numerical results in Section \ref{Section3}.

\section{The Neumann Neumann Method}\label{Section2}

In this section, we introduce the Neumann-Neumann method for the second order elliptic system \eqref{MF1}. For convenience we use the notation $u, v$ instead of $\bar{u}, \bar{v}$ and rewrite the system \eqref{MF1} as
\begin{equation}\label{modelproblem}
\begin{aligned}
\begin{bmatrix} 
I & -\delta\Delta \\
\epsilon^2\Delta-c^2 & I 
\end{bmatrix}
\begin{bmatrix} 
{u} \\
{v} 
\end{bmatrix}
 = \begin{bmatrix} 
f_{{u}} \\
f_{{v}} 
\end{bmatrix},
\quad \mbox{in}\,\ \Omega,
\end{aligned}
\end{equation}
together with the physical boundary condition $\mathcal{B}\begin{bmatrix} 
u \\
v 
\end{bmatrix}
=0.$

We present the NN method for the  problem
\eqref{modelproblem} on the spatial domain $\Omega$
with Neumann boundary conditions on $\partial\Omega$.
Suppose the spatial domain $\Omega$ is partitioned into two
non-overlapping subdomains $\Omega_{1}, \Omega_{2}$. We denote 
$u_{j}$ the restriction of the solution $u$ of \eqref{modelproblem} to $\Omega_{j}$ for $j=1, 2$ and set
$\Gamma:=\partial\Omega_{1}\cap\partial\Omega_{2}$.

The NN algorithm for the CH system \eqref{modelproblem} starts with initial guesses $g^{[0]}, h^{[0]}$ along the interface $\Gamma$ and solve for $k = 1, 2, ...$
\begin{equation}\label{NNCH}
\begin{aligned}
&\left\{
\begin{aligned}
\begin{bmatrix} 
I & -\delta\Delta \\
\epsilon^2\Delta-c^2 & I 
\end{bmatrix}
\begin{bmatrix} 
u_j^{[k]} \\
v_j^{[k]} 
\end{bmatrix}
  = \begin{bmatrix} 
f_u \\
f_v 
\end{bmatrix},\,\ \text{in} \,\ \Omega_j, \\
\mathcal{B}\begin{bmatrix} 
u_j^{[k]} \\
v_j^{[k]} 
\end{bmatrix}
 = 0,\,\ \mbox{on}\,\ \partial\Omega_j\cap\partial\Omega,\\
\begin{bmatrix} 
u_j^{[k]} \\
v_j^{[k]} 
\end{bmatrix}
 = 
\begin{bmatrix} 
g^{[k-1]} \\
h^{[k-1]} 
\end{bmatrix},
\,\ \mbox{on}\,\ \Gamma,\\
\end{aligned}\right.
&
\left\{\begin{aligned}
\begin{bmatrix} 
I & -\delta\Delta \\
\epsilon^2\Delta-c^2 & I 
\end{bmatrix}
\begin{bmatrix} 
\phi_j^{[k]} \\
\psi_j^{[k]} 
\end{bmatrix}
  = 0, \,\ \text{in} \,\ \Omega_j, \\
\mathcal{B}\begin{bmatrix} 
\phi_j^{[k]} \\
\psi_j^{[k]} 
\end{bmatrix}
 = 0,\,\ \mbox{on}\,\ \partial\Omega_j\cap\partial\Omega,\\
\frac{\partial}{\partial n}\begin{bmatrix} 
\phi_j^{[k]} \\
\psi_j^{[k]}
\end{bmatrix} =
\frac{\partial}{\partial n}\begin{bmatrix} 
u_1^{[k]} - u_2^{[k]} \\
v_1^{[k]} - v_2^{[k]} 
\end{bmatrix},
\,\ \mbox{on}\,\ \Gamma.\\
\end{aligned}\right. \\
\end{aligned}
\end{equation}
Then we update the interface trace by
\[
\begin{bmatrix} 
g^{[k]} \\
h^{[k]} 
\end{bmatrix}
 = \begin{bmatrix} 
g^{[k-1]} \\
h^{[k-1]}
\end{bmatrix} -
  \theta \begin{bmatrix} 
\phi_1^{[k]} - \phi_2^{[k]} \\
\psi_1^{[k]} - \psi_2^{[k]} 
\end{bmatrix}_{\big|_{\Gamma}},
\]
where $\theta\in(0, 1)$ is the relaxation parameter. For the convergence analysis we consider the error equation, which is a homogeneous version of \eqref{NNCH} by the virtue of linearity. The goal of our study is to analyze how the error $g^{[k]}, h^{[k]}$ converges to zero as $k\rightarrow\infty$.
\subsection{Convergence analysis in 1D}
Let the spatial domain $\Omega = (-a, b)$ is decomposed into $\Omega_1 = (-a, 0)$ and $\Omega_2 = (0, b)$ with interface $\Gamma=\{0\}$. We consider the error equation described in \eqref{NNCH} by taking $f_u = 0, f_v = 0$. The system \eqref{NNCH} can be rewritten in 1D as:
\begin{equation}\label{NNerr1}
\begin{aligned}
&\left\{
\begin{aligned}
A E_j^{[k]} & = 0, \quad \quad \text{in} \,\ \Omega_j \\  
 E_j^{[k]} & = \begin{bmatrix} 
 g^{[k-1]} \\
 h^{[k-1]} 
\end{bmatrix}, \quad \text{on} \,\ \Gamma \\
\frac{\partial}{\partial x}\ E_j^{[k]} & = 0, \quad \quad \text{on} \,\ \partial\Omega_j \cap\partial\Omega\\
  \end{aligned}\right.
&
\left\{\begin{aligned}
A F_j^{[k]} & = 0, \quad \quad \text{in} \,\ \Omega_i \\
\frac{\partial}{\partial x}\ F_j^{[k]} & = \frac{\partial}{\partial x}
\left[E_1^{[k]} - E_2^{[k]} \right]
, \quad \text{on} \,\ \Gamma \\
\frac{\partial}{\partial x}\ F_j^{[k]} & = 0, \quad \text{on} \,\ \partial\Omega_j\cap\partial\Omega\\
  \end{aligned}\right. \\
\end{aligned}
\end{equation}
with the update condition
\begin{equation}\label{update1d}
\begin{bmatrix} 
g^{[k]} \\
h^{[k]} 
\end{bmatrix} 
= 
\begin{bmatrix} 
 g^{[k-1]} \\
 h^{[k-1]} 
\end{bmatrix}
- 
\theta
\left[F_1^{[k]} - F_2^{[k]}\right]_{\big|_{\Gamma}},
\end{equation}
where
\begin{equation*}
 A = 
\begin{bmatrix} 
1 & -\delta\frac{d^2}{dx^2} \\
\epsilon^2\frac{d^2}{dx^2} - c^2 & 1 
\end{bmatrix},\quad
 E_j^{[k]} = \begin{bmatrix} 
 u_j^{[k]}(x) \\
 v_j^{[k]}(x)
\end{bmatrix},\quad
 F_j^{[k]} = \begin{bmatrix} 
 \phi_j^{[k]}(x) \\
 \psi_j^{[k]}(x)
\end{bmatrix},
\end{equation*}
where $E_j, F_j$ are the subdomain solutions in $\Omega_j$ for $j=1,2$ at the $k-$th iteration.

To analyze the convergence of the algorithm \eqref{NNCH}, we need to determine the subdomain solutions by solving the following algebraic equations
\begin{equation}\label{Algebraiceq}
AE_j =0, j=1,2.
\end{equation}
We assume the solution of the equation \eqref{Algebraiceq} is of the following form
\begin{equation}
E_j=\Psi_j e^{\xi x}, \text{where}\; \xi\; \text{is parameter to be determined.}
\end{equation}
Inserting the above expression of $E_j$ into equation \eqref{Algebraiceq} and using the non-vanishing property of exponential term, we get
\begin{equation}\label{coefficientmatrix}
\begin{bmatrix} 
1 & -\delta\xi^2\\
\epsilon^2\xi^2-c^2 & 1 
\end{bmatrix}
\Psi_j = 0.
\end{equation}
Now the equation \eqref{coefficientmatrix} has non-trivial solutions only if the coefficient matrix is singular, i.e., the determinant of the coefficient matrix 
is zero, 
\begin{equation*}
\det
\begin{bmatrix} 
1 & -\delta\xi^2\\
\epsilon^2\xi^2-c^2 & 1 
\end{bmatrix}
 = 0.
\end{equation*}
Solving this equation yields
$\xi_{1,2} = \pm\sqrt{\lambda_1},\quad \xi_{3,4} = \pm\sqrt{\lambda_2}$, where $\lambda_{1,2}$ are given by 
\begin{equation*}
\lambda_{1,2} = \frac{c^2\delta\pm \sqrt{c^4\delta^2-4\epsilon^2\delta}}{2\epsilon^2\delta}
\end{equation*}
respectively. Thus the solution of \eqref{Algebraiceq} has the following form 
\begin{equation*}\label{generalsol}
E_j^{[k]} = \sum_{l=1} ^4 \zeta_{j,l}^{[k]}\mu_l e^{\xi_l x}, 
\end{equation*}
where for each $l$, $\zeta_{j,l}$ are constant weights for $j=1,2$ and $\mu_l$ is the  eigenvector of the coefficient matrix in  \eqref{coefficientmatrix} associated to the eigenvalue zero for $\xi = \xi_l$, and is explicitly given by 
\begin{equation*}
\mu_{1} = \mu_{2} = \begin{bmatrix} 
\delta\lambda_1 \\
1 
\end{bmatrix},\quad \mu_{3} = \mu_{4} = \begin{bmatrix} 
\delta\lambda_2\\
1 
\end{bmatrix}.
\end{equation*}
Using the transmission conditions on the interface $\Gamma$ and the physical boundary conditions on $\partial\Omega_{i}\backslash\Gamma$, we can get rid of the constants $\zeta_{j,l}$ for each subdomain. Set $\sigma_{i,x}:=\sinh(\xi_i x), \sigma_{i,a}:=\sinh(\xi_i a), \sigma_{i,b}:=\sinh(\xi_i b), \gamma_{i,x}:=\cosh(\xi_i x),  \gamma_{i,a}:=\cosh(\xi_i a),  \gamma_{i,b}:=\cosh(\xi_i b)$ for $i=1,3$ and 
$\eta_1 = \frac{g^{k-1}-\delta\lambda_2 h^{k-1}}{\delta\lambda}, \eta_2 = \frac{g^{k-1}-\delta\lambda_1 h^{k-1}}{-\delta\lambda}$ with
$\lambda = \lambda_1 - \lambda_2$, and $\nu_i = \frac{\sigma_{i,a}}{\gamma_{i,a}} + \frac{\sigma_{i,b}}{\gamma_{i,b}}$ for $i = 1, 3$. We have the subdomain solution at $k-$th iteration for Dirichlet step
\[
E_1^{[k]} = 
\begin{bmatrix} 
\mu_1 & \mu_3
\end{bmatrix} 
\begin{bmatrix} 
\gamma_{1,x} + \frac{\sigma_{1,a} \sigma_{1,x}}{\gamma_{1,a}} & 0 \\
0 &  \gamma_{3,x} + \frac{\sigma_{3,a} \sigma_{3,x}}{\gamma_{3,a}}
\end{bmatrix}
\begin{bmatrix} 
\eta_1 \\
\eta_2
\end{bmatrix},
\]

\[
E_2^{[k]} = 
\begin{bmatrix} 
\mu_1 & \mu_3
\end{bmatrix} 
\begin{bmatrix} 
\gamma_{1,x} - \frac{\sigma_{1,b} \sigma_{1,x}}{\gamma_{1,b}} & 0 \\
0 &  \gamma_{3,x} - \frac{\sigma_{3,b} \sigma_{3,x}}{\gamma_{3,b}}
\end{bmatrix}
\begin{bmatrix} 
\eta_1 \\
\eta_2 
\end{bmatrix},
\]
and for Neumann step we solve $AF_j = 0$ for $j=1,2$ with physical boundary conditions on $\partial\Omega_{i}\backslash\Gamma$ and imposed transmission condition on the interface $\Gamma$ and get the subdomain solution at $k-$th iteration,
\[
F_1^{[k]} = 
\begin{bmatrix} 
\mu_1 & \mu_3
\end{bmatrix} 
\begin{bmatrix} 
\sigma_{1,x}\nu_1  + \frac{\gamma_{1,x}\gamma_{1,a}\nu_1}{\sigma_{1,a}} & 0 \\
0 & \sigma_{3,x}\nu_3  + \frac{\gamma_{3,x}\gamma_{3,a}\nu_3}{\sigma_{3,a}}
\end{bmatrix}
\begin{bmatrix} 
\eta_1 \\
\eta_2 
\end{bmatrix},
\]

\[
F_2^{[k]} = 
\begin{bmatrix} 
\mu_1 & \mu_3
\end{bmatrix} 
\begin{bmatrix} 
\sigma_{1,x}\nu_1  - \frac{\gamma_{1,x}\gamma_{1,b}\nu_1}{\sigma_{1,b}} & 0 \\
0 & \sigma_{3,x}\nu_3  - \frac{\gamma_{3,x}\gamma_{3,b}\nu_3}{\sigma_{3,b}}
\end{bmatrix}
\begin{bmatrix} 
\eta_1 \\
\eta_2 
\end{bmatrix}.
\] 
Using the subdomain solutions in \eqref{update1d} we get the recurrence relation
\begin{equation}\label{NN1diteration}
 \begin{bmatrix} 
g^{[k]} \\
h^{[k]} 
\end{bmatrix} 
= \mathbb{P}
\begin{bmatrix} 
 g^{[k-1]} \\
 h^{[k-1]}
\end{bmatrix},
\end{equation}
where the iteration matrix $\mathbb{P}$ is given by
\begin{equation*}
 \mathbb{P} = 
\begin{bmatrix} 
1 - \theta\frac{\lambda_1\rho_1 - \lambda_2\rho_2}{\lambda} & \theta\frac{\delta\lambda_1\lambda_2(\rho_1 - \rho_2)}{\lambda} \\
\\
\theta\frac{\rho_2 - \rho_1}{\delta\lambda} & 1 - \theta\frac{\lambda_1\rho_2 - \lambda_2\rho_1}{\lambda}
\end{bmatrix} ,
\end{equation*}
with $ \rho_1 = 2 + \frac{\sigma_{1,a}\gamma_{1,b}}{\gamma_{1,a}\sigma_{1,b}} + \frac{\sigma_{1,b}\gamma_{1,a}}{\gamma_{1,b}\sigma_{1,a}},\; \rho_2 = 2 + \frac{\sigma_{3,a}\gamma_{3,b}}{\gamma_{3,a}\sigma_{3,b}} + \frac{\sigma_{3,b}\gamma_{3,a}}{\gamma_{3,b}\sigma_{3,b}}$.

\begin{theorem}[Convergence of NN for symmetric case]
When the subdomains are of same size, $a=b$, the NN algorithm for two subdomains converges linearly for $0<\theta<1/2, \theta\neq\ 1/4$. Moreover for $\theta = 1/4$, it converges in two iterations.
\end{theorem}
\begin{proof}
If $a=b$, then \eqref{NN1diteration} becomes
\begin{equation*}
\begin{bmatrix} 
g^{[k]} \\
h^{[k]} 
\end{bmatrix} 
=  \begin{bmatrix} 
1-4\theta & 0 \\
0 & 1-4\theta 
\end{bmatrix}
\begin{bmatrix} 
g^{[k-1]} \\
h^{[k-1]} 
\end{bmatrix},
\end{equation*}
which means the method converges to the exact solution in two iterations for $\theta=1/4$. Also, the algorithm diverges if $\vert 1-4\theta\vert\geq 1$, i.e if $\theta\geq 1/2$. Thus the convergence is linear for $0<\theta<1/2, \theta\neq 1/4$.
\end{proof}

We now study the convergence results for unequal sudomains with the special value of $\theta$ being $1/4$. Before going to the main result, we prove the following Lemma, which is needed for future analysis.

\begin{lemma}\label{NNlemma}
The function $f(t)=\frac{\sinh^2((a - b)t)}{\sinh (2at) \sinh (2bt)}$ with $a, b > 0$ has the following properties
\begin{enumerate}
\item $\forall t > 0, 0 < f(t) < \frac{(a - b)^2}{4ab}$.\label{item:one}
\item $\forall t > 0, f(t)$ is a monotonically decreasing function.\label{item:second}
\end{enumerate}
\end{lemma}
\begin{proof}
Clearly $f(t)>0, \forall t > 0$. We prove \eqref{item:second} and the upper bound estimate in \eqref{item:one} will follow from that. Taking the derivative of logarithm of $f(t)$ with respect to $t$ we get
$$\frac{d}{dt}(\ln(f(t))) = 2(a-b)\coth((a-b)t) - 2a\coth(2at) - 2b\coth(2bt).$$
Now $f'(t)<0$ iff $(a-b)\coth((a-b)t)< a\coth(2at) + b\coth(2bt)$, which is true, since the function $t\coth(t)$ is strictly convex and increasing for $t > 0$. This completes the result.
\end{proof}

\begin{theorem}[Convergence of NN for non-symmetric case]\label{NN1dconv}
For $\theta = 1/4$ the error of the NN method for two subdomains satisfies,
\[
\parallel g^{[k]} \parallel_{L^{\infty}(\Gamma)} \leq
\begin{cases}
 \left(\frac{(a - b)^2}{4ab}\right)^k \max\left\{\parallel g^{[0]}\parallel_{L^{\infty}(\Gamma)}, \parallel h^{[0]}\parallel_{L^{\infty}(\Gamma)}\right\}, &\text{if}\; \delta > \frac{4\epsilon^2}{c^4},\\
\left(\frac{(a - b)^2}{2ab}\right)^k \max\left\{\parallel g^{[0]}\parallel_{L^{\infty}(\Gamma)}, \parallel h^{[0]}\parallel_{L^{\infty}(\Gamma)}\right\}, &\text{if}\; \delta < \frac{4\epsilon^2}{c^4}.
 \end{cases}
\]
\end{theorem}
\begin{proof}
When $\delta > \frac{4\epsilon^2}{c^4}$, it is clear that $\lambda_{1,2}$ are real and positive, so are $\xi_{1,3}$. The iteration matrix $\mathbb{P}$ in \eqref{NN1diteration} has eigenvalues $(1 - \frac{\rho_1}{4}), (1 - \frac{\rho_2}{4})$. Hence, the convergence is reached iff the spectral radius $\rho(\mathbb{P}) = \max\{\vert 1 - \frac{\rho_1}{4}\vert, \vert 1 - \frac{\rho_2}{4}\vert\}$ is less than one. Upon simplification, we get the expression of the eigenvalues as: $-\frac{\sinh^2((a - b)\xi_i)}{\sinh (2a\xi_i) \sinh (2b\xi_i)}$, for $i = 1, 3$. Using the Lemma \eqref{NNlemma} we have the required estimate.\\
For $\delta < \frac{4\epsilon^2}{c^4}$, $\lambda_1$ and  $\lambda_2$ are complex conjugate numbers. We denote $\lambda_{1,2}$ as $\lambda_{1,2}=\lambda_{\Re} \pm i\lambda_{\Im}, i = \sqrt{-1}$, where
\[
\lambda_{\Re}=\frac{ c^2 }{2\epsilon^2},\; \lambda_{\Im}=\frac{\sqrt{4\delta\epsilon^2 - \delta^2 c^4} }{2\delta\epsilon^2}.
\]
 Then $\xi_{1,3}$ takes the form $\xi_1=\xi_{\Re} + i\xi_{\Im}, \xi_3=\xi_{\Re} - i\xi_{\Im}$, where 
 \[
\xi_{\Re}=\sqrt{\frac{\lambda_{\Re}+\sqrt{\lambda_{\Re}^2+\lambda_{\Im}^2}}{2}}, \; \xi_{\Im}=\sqrt{\frac{-\lambda_{\Re}+\sqrt{\lambda_{\Re}^2+\lambda_{\Im}^2}}{2}}.
\]
Clearly, $\xi_{\Re}, \xi_{\Im}$ are positive number. Now if we take the modulus of the complex eigenvalues of the iteration matrix $\mathbb{P}$, we have
\begin{equation*}
\begin{aligned}
\left|-\frac{\sinh^2((a - b)\xi_i)}{\sinh (2a\xi_i) \sinh (2b\xi_i)}\right| & \leq &
\frac{\sinh^2((a-b)\xi_{\Re}) + \sin^2((a-b)\xi_{\Im})}{\sinh(2a\xi_{\Re}) \sinh(2b\xi_{\Re})} \\
 & \leq & \frac{2\sinh^2((a-b)\xi_{\Re})}{\sinh(2a\xi_{\Re}) \sinh(2b\xi_{\Re})} \leq \left(\frac{(a - b)^2}{2ab}\right),
\end{aligned}
\end{equation*} 
the second inequality follows from $\sin^2((a-b)\xi_{\Im}) < \sinh^2((a-b)\xi_{\Im}) < \sinh^2((a-b)\xi_{\Re})$, as $\xi_{\Im} < \xi_{\Re}$, and the last inequality follows from the Lemma \eqref{NNlemma}. Hence the estimate.
\end{proof}
\begin{remark}
The above linear estimate will imply convergence if $\frac{(a - b)^2}{4ab}<1$, i.e., $q^2-6q+1<0$, where $q=a/b$. Hence $q$ must belongs to $\left(\frac{6-\sqrt{32}}{2}, \frac{6+\sqrt{32}}{2}\right)$ when $\delta > \frac{4\epsilon^2}{c^4}$ and for $\delta < \frac{4\epsilon^2}{c^4}$, $q$ is in $(2-\sqrt{3}, 2+\sqrt{3})$. 
\end{remark}
\subsection{Convergence analysis in 2D}
We now analyze the convergence of the NN method for a decomposition with two subdomains in 2D. Let the domain $\Omega = (-a, b)\times(0,L)$ be partitioned into two rectangles, given by $\Omega_1 = (-a, 0)\times(0,L)$ and $\Omega_2 = (0, b)\times(0,L)$.
We analyze this 2D case by transforming it into a collection of 1D problems using the Fourier sine transform along $y$-direction.
Expanding the solution $u_i^{[k]}, v_i^{[k]}, \phi_{i}^{[k]}, \psi_{i}^{[k]}$ for $i=1, 2$ in a Fourier sine series along $y$-direction yields
\[
\begin{array}{cc}
u_{i}^{[k]}(x,y)=\sum_{m\geq1}\hat u_{i}^{[k]}(x,m)\sin(\frac{m\pi y}{L}),\; &
v_{i}^{[k]}(x,y)=\sum_{m\geq1}\hat v_{i}^{[k]}(x,m)\sin(\frac{m\pi y}{L}),\\
\phi_{i}^{[k]}(x,y)=\sum_{m\geq1}\hat \phi_{i}^{[k]}(x,m)\sin(\frac{m\pi y}{L}),\;&
\psi_{i}^{[k]}(x,y)=\sum_{m\geq1}\hat \psi_{i}^{[k]}(x,m)\sin(\frac{m\pi y}{L}).
\end{array}
\]
After a Fourier sine transform, the NN algorithm \eqref{NNCH} for the error equations for 2D CH equation becomes
\begin{equation}\label{NNerr2}
\begin{aligned}
&\left\{
\begin{aligned}
\hat A \hat E_1^{[k]} & = 0, \quad \quad \text{in} \,\ (-a, 0), \\ 
\hat E_1^{[k]} & = \begin{bmatrix} 
\hat g^{[k-1]} \\
\hat h^{[k-1]} 
\end{bmatrix}, \quad \text{at} \,\ x = 0, \\
\frac{\partial}{\partial x}\ \hat E_1^{[k]} & = 0, \quad \quad \text{at} \,\ x = -a,\\
  \end{aligned}\right.
&
\left\{\begin{aligned} 
\hat AE_2^{[k]} & = 0,  \quad \quad \text{in} \,\ (0, b),\\
\hat E_2^{[k]} & = \begin{bmatrix} 
\hat g^{[k-1]} \\
\hat h^{[k-1]} 
\end{bmatrix}, \quad \text{at} \,\ x = 0, \\
\frac{\partial}{\partial x}\ \hat E_2^{[k]} & = 0, \quad \quad \text{at} \,\ x=b,\\ 
  \end{aligned}\right. \\
\end{aligned}
\end{equation}  
   and
\begin{equation}\label{NNerr2}
\begin{aligned}
&\left\{
\begin{aligned}
\hat A \hat F_1^{[k]} & = 0, \quad \quad \text{in} \,\ (-a, 0), \\
\frac{\partial}{\partial x}\ \hat F_1^{[k]} & = \frac{\partial}{\partial x}\left[\hat E_1^{[k]} - \hat E_2^{[k]}\right],\,\text{at} \,\ x = 0,\\
\frac{\partial}{\partial x}\ \hat F_1^{[k]} & = 0, \quad \quad \text{at} \,\ x = -a,\\
 \end{aligned}\right.
&
\left\{\begin{aligned}  
\hat A \hat F_2^{[k]} & = 0,  \quad \quad \text{in} \,\ (0, b),\\
 \frac{\partial}{\partial x}\ \hat F_2^{[k]} & = \frac{\partial}{\partial x}\left[\hat E_1^{[k]} - \hat E_2^{[k]}\right], \, \text{at} \,\ x =0,  \\
\frac{\partial}{\partial x}\ \hat F_2^{[k]} & = 0, \quad \quad \text{at} \,\ x=b,\\ 
 \end{aligned}\right. \\
\end{aligned}
\end{equation}    
and the update condition becomes
\[ 
 \begin{bmatrix} 
\hat g^{[k]} \\
\hat h^{[k]} 
\end{bmatrix} 
= 
\begin{bmatrix} 
\hat g^{[k-1]} \\
\hat h^{[k-1]} 
\end{bmatrix}
 - 
\theta
\left[\hat F_1^{[k]} - \hat F_2^{[k]}\right]_{\big|\{x=0\}},  
\]
where
\[
\hat A = 
\begin{bmatrix} 
1 & -\delta(\frac{d^2}{dx^2} - p_m^2) \\
\epsilon^2(\frac{d^2}{dx^2} - p_m^2) - c^2 & 1 
\end{bmatrix},\;
\hat E_j^{[k]} = 
\begin{bmatrix} 
\hat u_j^{[k]}(x,m) \\
\hat v_j^{[k]}(x,m)
\end{bmatrix},\;
\hat F_j^{[k]} = 
\begin{bmatrix} 
\hat \phi_j^{[k]}(x,m) \\
\hat \psi_j^{[k]}(x,m)
\end{bmatrix},
\]
where $\hat E_j^{[k]}(x,m), \hat F_j^{[k]}(x,m)$ denote the Fourier sine coefficients of $ E_j^{[k]}(x,y)$ and $ F_j^{[k]}(x,y)$, and $p_m^2 = \frac{\pi ^2 m^2}{L^2}$.
We perform similar calculation as in the case of one dimensional analysis and get the recurrence relation as:
\begin{equation}\label{NNreccurence2d}
\begin{bmatrix} 
\hat g^{[k]} \\
\hat h^{[k]} 
\end{bmatrix} 
= \hat {\mathbb{P}}
\begin{bmatrix} 
\hat g^{[k-1]} \\
\hat h^{[k-1]} 
\end{bmatrix},
\end{equation}
where $\hat{\mathbb{P}}$ is the iteration matrix given by   
\[
\hat{\mathbb{P}} = 
\begin{bmatrix} 
1 - \theta\frac{\lambda_1\rho_1 - \lambda_2\rho_2}{\lambda} & \theta\frac{\delta\lambda_1\lambda_2(\rho_1 - \rho_2)}{\lambda} \\
\\
\theta\frac{\rho_2 - \rho_1}{\delta\lambda} & 1 - \theta\frac{\lambda_1\rho_2 - \lambda_2\rho_1}{\lambda}
\end{bmatrix} 
\]
with the expressions $\rho_1, \rho_2$ exactly as defined earlier, except that $\lambda_i$ is replaced by $\lambda_i + p_m^2$ in the expression of $\xi_{1,2}$ and $\xi_{3,4}$.
\begin{theorem}[Convergence of NN in 2D]
When the subdomains are of same size, $a=b$, the NN algorithm converges linearly for $0<\theta<1/2, \theta\neq\ 1/4$, and for $\theta = 1/4$, it converges in two iterations.
When the subdomains are unequal, for $\theta = 1/4$ the error of the NN method for two subdomains satisfies,
\[
\parallel g^{[k]} \parallel_{L^{2}(\Gamma)} \leq
\begin{cases}
 \left(\frac{(a - b)^2}{4ab}\right)^k \max\left\{\parallel g^{[0]}\parallel_{L^{2}(\Gamma)}, \parallel h^{[0]}\parallel_{L^{2}(\Gamma)}\right\}, &\text{if}\; \delta > \frac{4\epsilon^2}{c^4},\\
\left(\frac{(a - b)^2}{2ab}\right)^k \max\left\{\parallel g^{[0]}\parallel_{L^{2}(\Gamma)}, \parallel h^{[0]}\parallel_{L^{2}(\Gamma)}\right\}, &\text{if}\; \delta < \frac{4\epsilon^2}{c^4}.
 \end{cases}
\]
\end{theorem}
\begin{proof}
If $a=b$, then \eqref{NNreccurence2d} yields 
\begin{equation*}
\begin{bmatrix} 
\hat g^{[k]} \\
\hat h^{[k]} 
\end{bmatrix} 
=  \begin{bmatrix} 
1-4\theta & 0 \\
0 & 1-4\theta 
\end{bmatrix}
\begin{bmatrix} 
\hat g^{[k-1]} \\
\hat h^{[k-1]} 
\end{bmatrix},
\end{equation*}
which upon back Fourier transform shows that the method converges to the exact solution in two iterations for $\theta=1/4$. Also, the algorithm diverges if $\vert 1-4\theta\vert\geq 1$, i.e if $\theta\geq 1/2$. Thus the convergence is linear for $0<\theta<1/2, \theta\neq 1/4$.

Now in the case of unequal subdomain with $\delta > \frac{4\epsilon^2}{c^4}$, it is clear that $\lambda_{1,2}$ are real and positive, so are $\xi_{1,3}$. The iteration matrix $\hat{\mathbb{P}}$ in \eqref{NNreccurence2d} has spectral radius $\rho(\hat{\mathbb{P}}) = \max\left\{\vert 1 - \frac{ \rho_1}{4}\vert, \vert 1 - \frac{ \rho_2}{4}\vert\right\}$. By the recurrence relation \eqref{NNreccurence2d} and using Parseval-Plancherel identity we get
\[
\begin{array}{cc}
\parallel g^{[k]}\parallel_{L^{2}(\Gamma)} \leq \rho(\hat{\mathbb{P}}) \max\left\{\parallel g^{[k-1]}\parallel_{L^{2}(\Gamma)}, \parallel h^{[k-1]}\parallel_{L^{2}(\Gamma)}\right\}, \\
\parallel h^{[k]}\parallel_{L^{2}(\Gamma)} \leq \rho(\hat{\mathbb{P}})\max \left\{\parallel g^{[k-1]}\parallel_{L^{2}(\Gamma)}, \parallel h^{[k-1]}\parallel_{L^{2}(\Gamma)}\right\}.
\end{array}
\]
We get the required result by estimating $\rho(\hat{\mathbb{P}})$ similar to Theorem \eqref{NN1dconv}. \\
The case $\delta < \frac{4\epsilon^2}{c^4}$ can be treated similarly as in Theorem \eqref{NN1dconv}.
\end{proof}
\section{Numerical Illustration}\label{Section3}
First, we present numerical results for the NN algorithm with small time window. The parameter $\epsilon $ is taken as $0.01$, except otherwise stated. The iterations start from a random initial guess and stop as the error reaches a tolerance of $10^{-6}$ in $L^{\infty} \;\text{or}\; L^2$ norm. 
The phase separation process is rapid in time, and consequently small time steps should be taken. We choose $\delta=1\times 10^{-6}$, which falls in the bracket $\delta< \frac{4\epsilon^2}{c^4}$. 
The iteration count for the NN method are shown in Table \ref{table5e} for 1D.
In Fig \ref{NNshorttime}, on the left we plot the error curves for NN method for different values of  $\theta$ for CH equation in 2D, and on the right we show the numerical error and theoretical estimates for $\theta=1/4$.
\vspace{-0.2cm}
\begin{table}
\begin{center}
\begin{tabular}{|p{.7cm}| p{.5cm}| p{.5cm}| p{.5cm}| p{.5cm}| p{.5cm}| p{.5cm}| p{.5cm}| p{.5cm}| p{.5cm}| }\hline
\diagbox{$h$}{$\theta$}&{ 0.l }&{0.2}&{\bf 0.25}&{0.3}&{0.4}&{ 0.5}&{0.6}&{0.7}&{0.8} \\
		\hline1/64&  35& 12& \bf 2& 12&  37& -& -& -& -\\
		\hline1/128& 35& 12& \bf 2& 12&  38& -& -& -& -\\
		\hline1/256& 36& 12& \bf 2& 12&  39& -& -& -& -\\
		\hline1/512& 37& 12& \bf 2& 12&  39& -& -& -& -\\
		\hline
	\end{tabular}
\end{center}
\caption{Number of iteration compared for unequal subdomain in 1D.}
\label{table5e}
\end{table}
\begin{figure}
    \centering
    \subfloat{{\includegraphics[width=5cm,height=3cm]{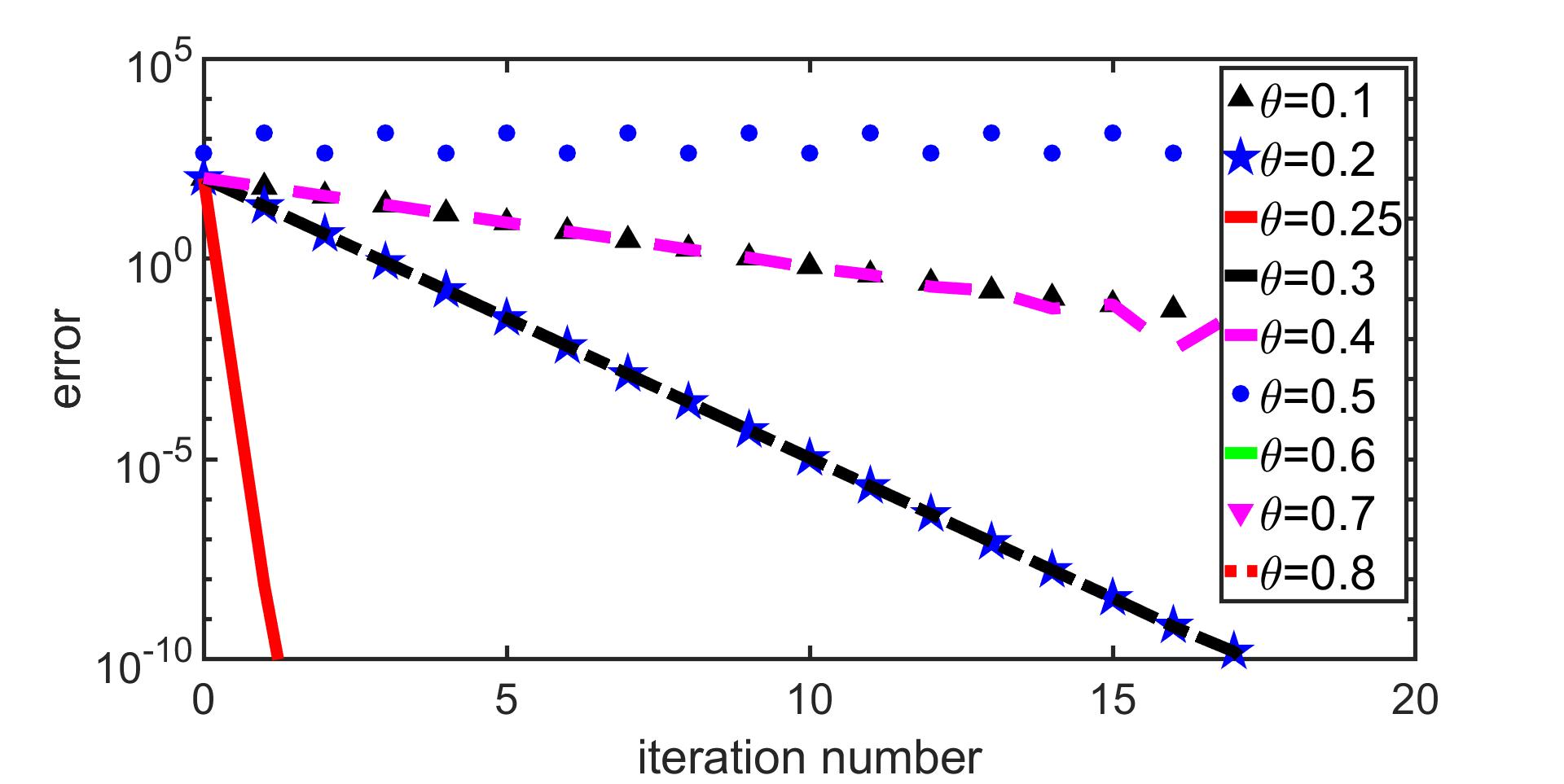}}}
    \subfloat{{\includegraphics[width=5cm,height=3cm]{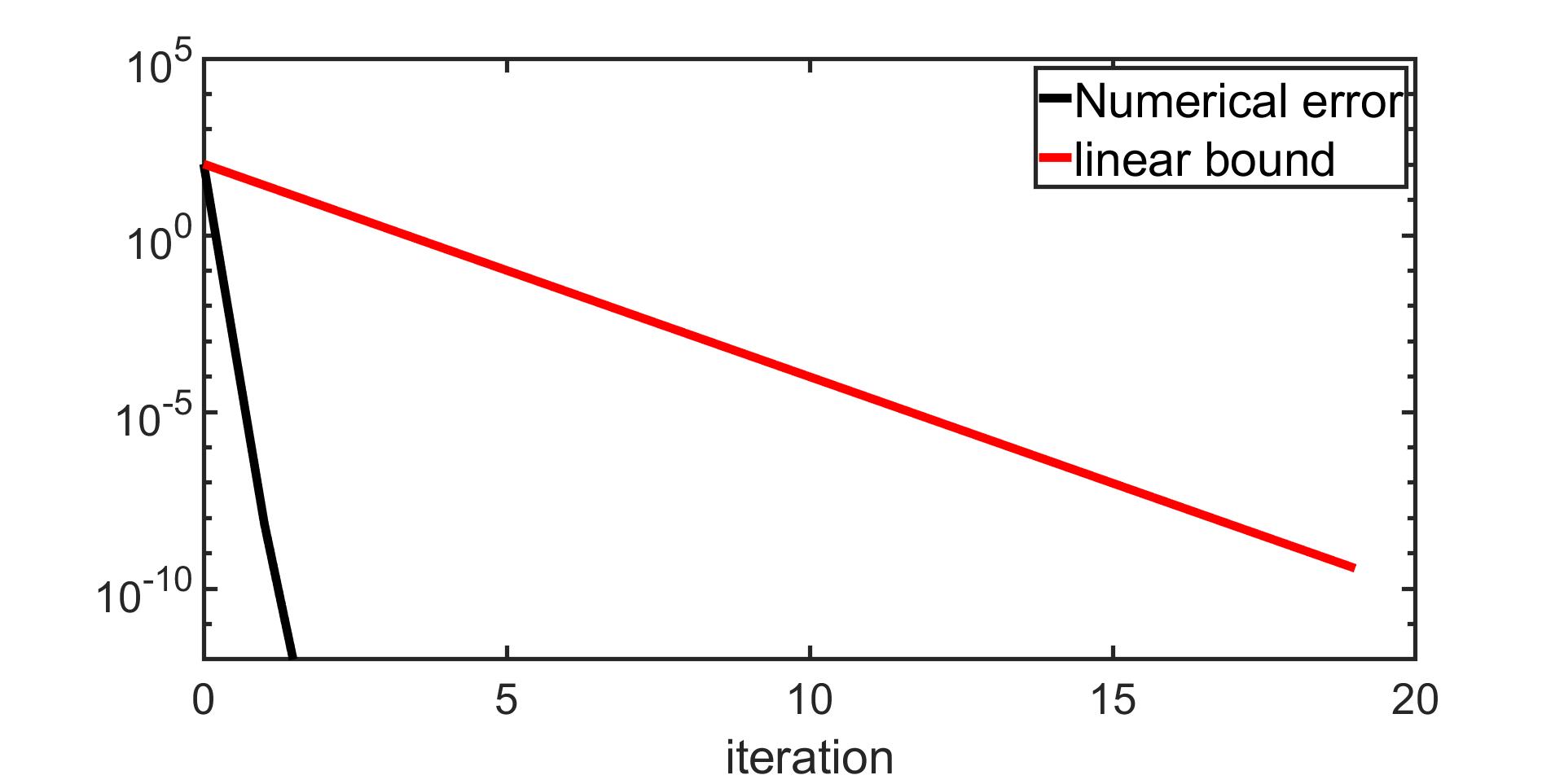} }}
    \caption{Iteration compared in 2D (left), and comparison of theoretical and numerical error (right)  with mesh size $h=1/64$.}
    \label{NNshorttime}
\end{figure}
\vspace{-1cm}
For the CH equation, phase coarsening stage is slow in time, and so one chooses relatively larger time step to reduce the total amount of calculation. We choose $\delta = 1\times 10^{-3}$ in our experiment that falls under $\delta > \frac{4\epsilon^2}{c^4}$. 
In this case, the iteration count for the NN method are shown in Table \ref{table5ef}.
 \begin{table}[t]
\begin{center}
\begin{tabular}{|p{.7cm}| p{.5cm}| p{.5cm}| p{.5cm}| p{.5cm}| p{.5cm}| p{.5cm}| p{.5cm}| p{.5cm}| p{.5cm}| }\hline
\diagbox{$h$}{$\theta$}&{ 0.l }&{0.2}&{\bf 0.25}&{0.3}&{0.4}&{ 0.5}&{0.6}&{0.7}&{0.8} \\
		\hline1/64&  36& 13& \bf 2& 13&  39& -& -& -& -\\
		\hline1/128& 37& 13& \bf 2& 13&  40& -& -& -& -\\
		\hline1/256& 38& 13& \bf 2& 13&  41& -& -& -& -\\
		\hline1/512& 38& 13& \bf 2& 14&  41& -& -& -& -\\
		\hline
	\end{tabular}
\end{center}
\caption{Number of iteration compared for unequal subdomain in 1D.}
\label{table5ef}
\end{table}
In Fig \ref{NNlongtime}, on the left we plot the error curves for NN method for different values of  $\theta$ for CH equation in 2D, and on the right we show the numerical error and theoretical estimates for $\theta=1/4$. 
\begin{figure}
    \centering
    \subfloat{{\includegraphics[width=5cm,height=3cm]{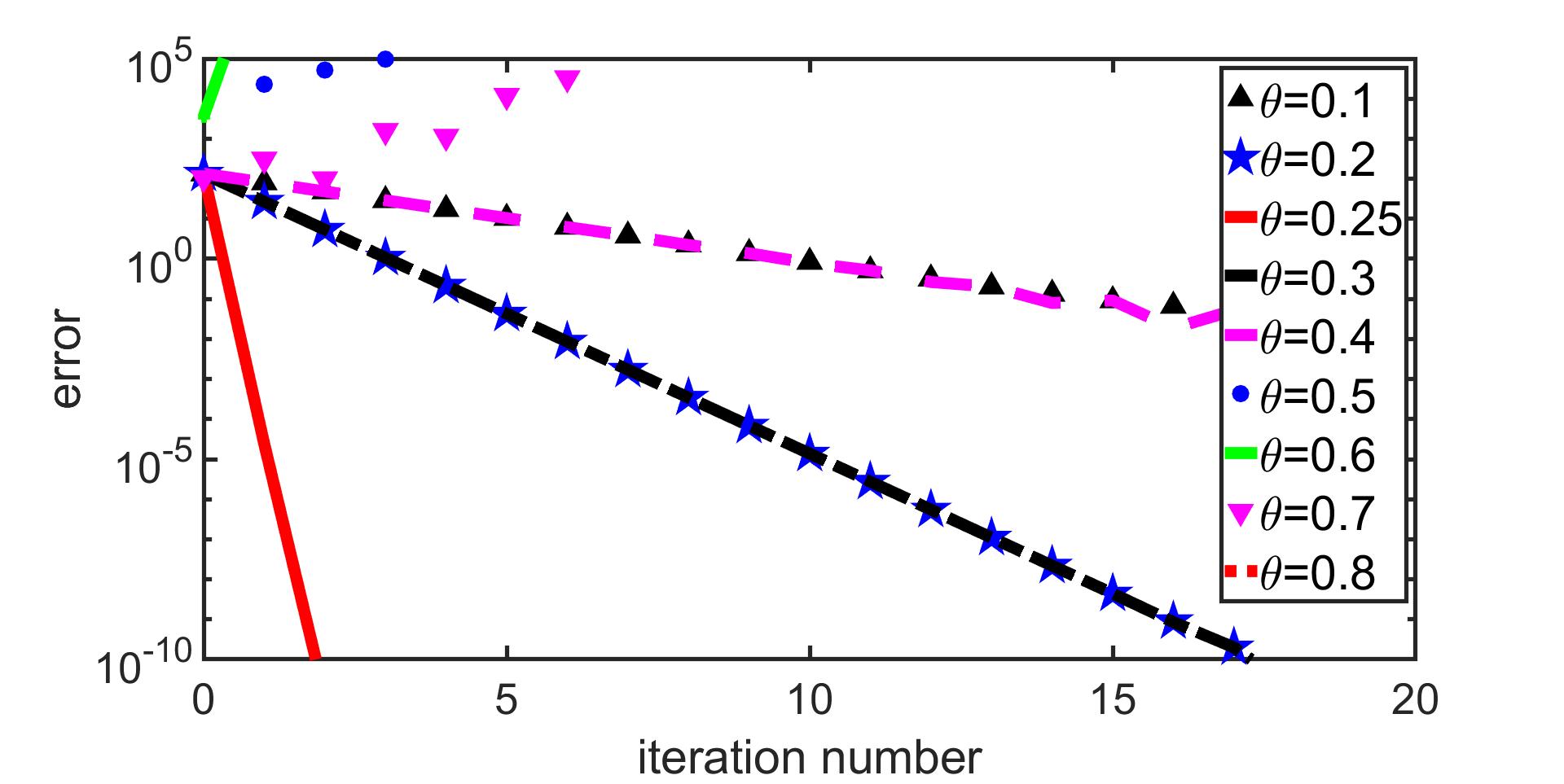}}}
    \subfloat{{\includegraphics[width=5cm,height=3cm]{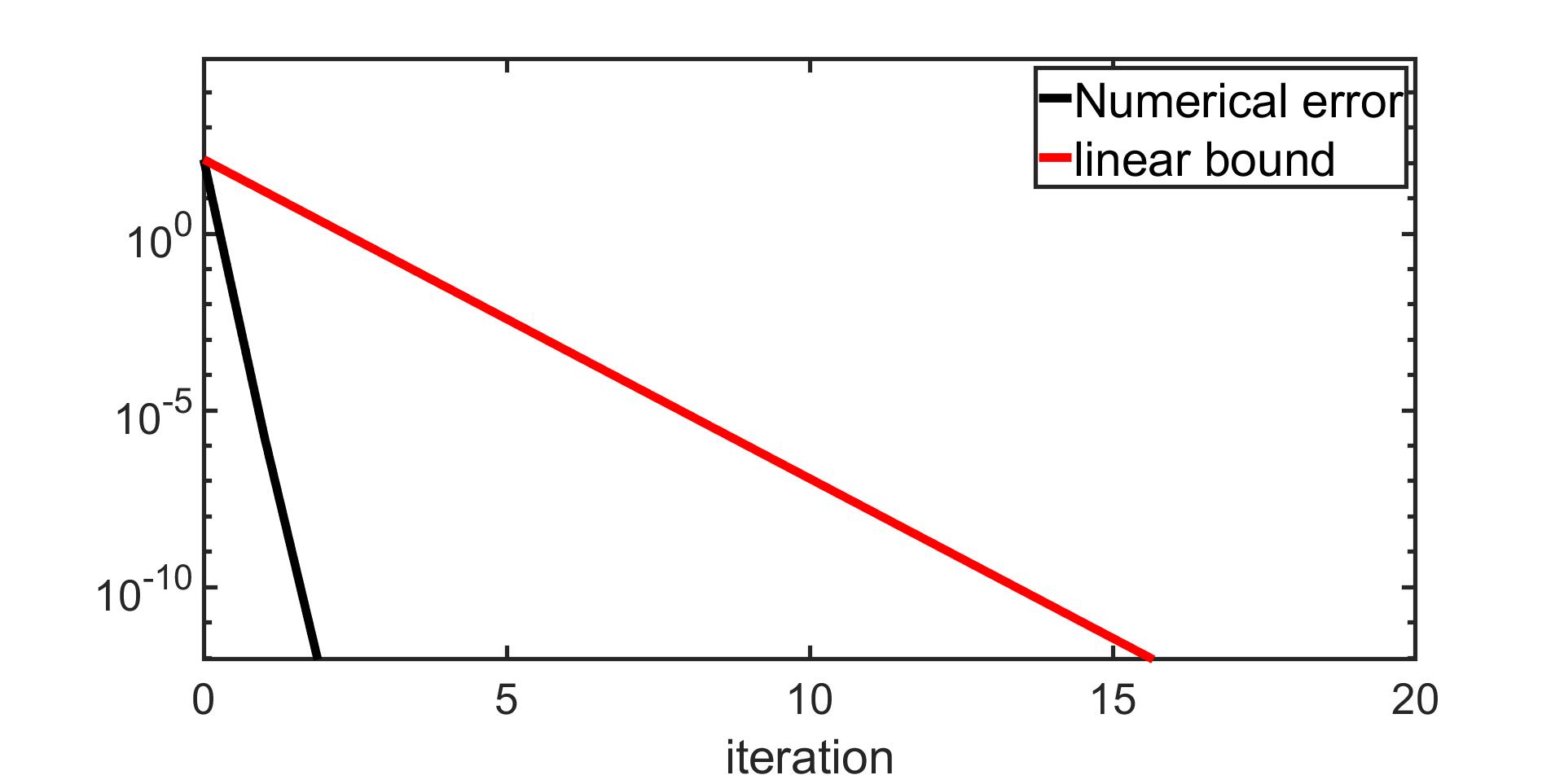} }}
    \caption{Iteration compared in 2D (left), and comparison of theoretical and numerical error (right)  with mesh size $h=1/64$.}
    \label{NNlongtime}
\end{figure}
\section{Conclusions}
We studied the Neumann-Neumann method for the CH equation for two subdomains. We proved convergence estimates for the case of one dimensional NN and also extended the analysis to the 2D CH equation using Fourier techniques. Lastly we validate our theoretical findings with numerical experiments.
\paragraph{Acknowledgement:} I wish to express my appreciation to Dr. Bankim C. Mandal for his constant support and stimulating suggestions and also like to thank the CSIR India for the financial assistance  and IIT Bhubaneswar for research facility.
\vspace{-.8cm}
\bibliographystyle{spmpsci}
\bibliography{author2}

\begin{thebibliography}{10}
\providecommand{\url}[1]{{#1}}
\providecommand{\urlprefix}{URL }
\expandafter\ifx\csname urlstyle\endcsname\relax
  \providecommand{\doi}[1]{DOI~\discretionary{}{}{}#1}\else
  \providecommand{\doi}{DOI~\discretionary{}{}{}\begingroup
  \urlstyle{rm}\Url}\fi

\bibitem{EsedoAndrea}
Bertozzi, A.L., Esedo\={g}lu, S., Gillette, A.: Inpainting of binary images
  using the {C}ahn-{H}illiard equation.
\newblock IEEE Trans. Image Process. \textbf{16}(1), 285--291 (2007)

\bibitem{BjWid}
Bj{\o}rstad, P.E., Widlund, O.B.: Iterative methods for the solution of
  elliptic problems on regions partitioned into substructures.
\newblock SIAM J. Numer. Anal. \textbf{23}(6), 1097--1120 (1986).
\newblock \doi{10.1137/0723075}

\bibitem{bourgat1988variational}
Bourgat, J.F., Glowinski, R., Le~Tallec, P., Vidrascu, M.: Variational
  formulation and algorithm for trace operation in domain decomposition
  calculations.
\newblock Ph.D. thesis, INRIA (1988)

\bibitem{Cahn}
Cahn, J.W.: On spinodal decomposition.
\newblock Acta Metall \textbf{9}(9), 795--801 (1961)

\bibitem{Hilliard}
Cahn, J.W., Hilliard, W.: Free energy of a nonuniform system. i. interfacial
  free energy.
\newblock J. Chem. Phys. \textbf{28}(2), 258--267 (1958)

\bibitem{David}
Eyre, D.J.: Unconditionally gradient stable time marching the {C}ahn-{H}illiard
  equation.
\newblock In: Computational and mathematical models of microstructural
  evolution ({S}an {F}rancisco, {CA}, 1998), \emph{Mater. Res. Soc. Sympos.
  Proc.}, vol. 529, pp. 39--46. MRS, Warrendale, PA (1998)

\bibitem{Eyre}
Eyre, D.J.: An unconditionally stable one-step scheme for gradient systems.
\newblock Unpublished article  (1998)

\bibitem{le1991domain}
Le~Tallec, P., De~Roeck, Y.H., Vidrascu, M.: Domain decomposition methods for
  large linearly elliptic three-dimensional problems.
\newblock Journal of Computational and Applied Mathematics \textbf{34}(1),
  93--117 (1991)

\bibitem{KimLee}
Lee, S., Lee, C., Lee, H.G., Kim, J.: Comparison of different numerical schemes
  for the cahn-hilliard equation.
\newblock J. KSIAM \textbf{17}(3), 197--207 (2013)

\bibitem{LionsI}
Lions, P.L.: On the {S}chwarz alternating method. {I}.
\newblock In: First {I}nternational {S}ymposium on {D}omain {D}ecomposition
  {M}ethods for {P}artial {D}ifferential {E}quations ({P}aris, 1987), pp.
  1--42. SIAM, Philadelphia, PA (1988)

\bibitem{TosWid}
Toselli, A., Widlund, O.B.: Domain Decomposition Methods, Algorithms and
  Theory, vol.~34.
\newblock Springer-Verlag, Berlin (2005)

\end{thebibliography}
\nocite{*}
\end{document}